\documentclass[10pt,leqno]{amsart}
\usepackage{graphicx}
\usepackage[T1]{fontenc}

\usepackage{amsmath}
\usepackage{amssymb}

\usepackage{comment}
\usepackage{mathtools}  
\usepackage{xfrac}
\usepackage{float}
\usepackage{verbatim}
\usepackage{faktor}
\usepackage{soul}
\usepackage{amsthm}
\usepackage{lipsum}
\usepackage{titleps}
\usepackage{hyperref}
\usepackage{cleveref}
\usepackage{bbold}
\usepackage[dvipsnames]{xcolor}

\usepackage{wrapfig} % wrapping
\usepackage{pifont} % za kvačicu
\usepackage{enumitem} % za numeriranje
\usepackage{graphicx} % za slike
\graphicspath{ {images/} }
\usepackage{pgf,tikz,pgfplots}
\pgfplotsset{compat=1.15}
\usepackage{mathrsfs}
\usetikzlibrary{arrows}
\usepackage{blindtext}
\usepackage{shapepar}
\usepackage{textcomp}
\usepackage{setspace}
\usepackage{multicol}

\newcommand\m[1]{\begin{bmatrix}#1\end{bmatrix}} 

\newcommand\sm[1]{\begin{bsmallmatrix}#1\end{bsmallmatrix}}
\newcommand\sabcd{\sm{a & b \\ c & d}}

\newcommand{\Z}{\mathbb{Z}}
\newcommand{\Q}{\mathbb{Q}}
\newcommand{\F}{\mathbb{F}}

\newcommand{\GL}{\textnormal{GL}}
\newcommand{\AGL}{\textnormal{AGL}}

\DeclareMathOperator{\Gal}{Gal}

\usepackage{fancyhdr,etoolbox}
\newtheorem{theorem}{Theorem}[]

\newtheorem{lemma}[theorem]{Lemma}

\newtheorem{proposition}[theorem]{Proposition}
\newtheorem{corollary}[theorem]{Corollary}

\newtheorem*{question}{Question}

\theoremstyle{definition}

\newtheorem{definition}[theorem]{Definition}

\theoremstyle{remark}

\newtheorem{remark}[theorem]{Remark}

\hypersetup{ colorlinks=true, linkcolor=black, filecolor=black, urlcolor=black }

\usepackage{lipsum}

\numberwithin{theorem}{section}

\begin{document}
\title{Degrees of isogenies over prime degree number fields of non-CM elliptic curves with rational $j$-invariant} %%%%%%%%%%%%
\author{Ivan Novak}
\date{\today}
\address{University of Zagreb, Bijeni\v{c}ka Cesta 30, 10000 Zagreb, Croatia}
\email{ivan.novak@math.hr}

\thanks{The author was supported by the project “Implementation of cutting-edge research and its application as part of the Scientific Center of Excellence for Quantum and Complex Systems, and Representations of Lie Algebras“, PK.1.1.02, European Union, European Regional Development Fund and by the Croatian Science Foundation under the project no. IP-2022-10-5008.}
\maketitle

\let\thefootnote\relax
%\footnotetext{MSC2020: Primary 00A05, Secondary 00A66.} %%%%%%%%%%

\begin{abstract}
We determine all possible degrees of cyclic isogenies of non-CM elliptic curves with rational $j$-invariant over number fields of degree $p$, where $p$ is an odd prime. The question had been answered for $p=2$, so this paper completes the classification in case when the degree of the number field is prime. 
\end{abstract} %%%%%%%%%

\bigskip

\section{Introduction}

\noindent Let $K$ be a number field. For a positive integer $n$, the $K$-rational points on the modular curve $X_0(n)$ correspond to elliptic curves together with a cyclic isogeny of degree $n$ defined over $K$. Thus, finding a point of some fixed degree $d$ on $X_0(n)$ is equivalent to finding an elliptic curve over a degree $d$ number field together with an $n$-isogeny defined over the same field.

A question which naturally arises is to find all possible degrees of isogenies which arise for elliptic curves belonging to some class. For example, one may look at the class of all elliptic curves defined over a degree $d$ number field for some fixed positive integer $d$. This is equivalent to finding all positive integers $n$ for which the curve $X_0(n)$ has a point of degree $d$. More generally, one may try finding all degree $d$ points on $X_0(n)$.

This is in general a difficult task for $d>1$. The case $d=1$, i.e. finding all rational points on $X_0(n)$, has been done by Mazur \cite{mazur} and Kenku (see e.g. \cite{Kenku}). The case $d>1$ is much more difficult. Some progress has been made in computing quadratic points on $X_0(n)$ (\cite{nikola}).

A similar question which is easier to answer, and which is more specific, is to describe all points on $X_0(n)$ of some degree $d$ for which the corresponding $j$-invariant is rational. More precisely, we are interested in the following question.

\begin{question}
Let $d$ be a fixed positive integer. Find all positive integers $n$ for which there is a non-CM elliptic curve $E$ with rational $j$-invariant and a cyclic isogeny of degree $n$ defined over some number field $K$ of degree $d$.
\end{question}

\begin{remark}
    We restrict our attention only to non-CM elliptic curves. In the CM case, an elliptic curve $E$ whose endomorphism ring is an order $\mathcal O$ in some imaginary quadratic field will have isogenies of prime degrees for all prime numbers which split over that order, and the isogenies will be defined over any field containing $\mathcal O$ and the field of definition of $E$. More on the CM case can be found in \cite{Bourdon-Clark}.
\end{remark}

We define an $n$-isogeny to be a cyclic isogeny of degree $n$.

For a positive integer $d$, denote by $\Psi_\Q(d)$ the set of all positive integers $n$ for which there is a non-CM elliptic curve with rational $j$-invariant and an $n$-isogeny defined over a degree $d$ extension of $\Q$. In other words, $\Psi_\Q(d)$ is precisely the set we are trying to determine.

By Mazur \cite{mazur} and Kenku \cite{Kenku}, we know $\Psi_\Q(1)$. In other words, we know all positive integers $n$ for which there is a non-CM elliptic curve $E/\Q$ with a rational $n$-isogeny. We have $$\Psi_\Q(1)=\{1,2,\ldots, 13, 15, 16, 17, 18, 21, 25, 37\}.$$

Vukorepa proved in \cite{Borna} that $\Psi_\Q(2)=\Psi_\Q(1) \cup \{20,24,32,36\}$. Some of the proofs from that paper also apply to our problem.

The purpose of this article is to answer the question when the degree of the extension is an odd prime. Our main results are the following. 

\begin{theorem}\label{kubni}
    There exists a non-CM elliptic curve with rational $j$-invariant and an $n$-isogeny defined over a cubic extension of $\Q$ if and only if \begin{itemize}
        \item $n \in \Psi_\Q(1)$,
        \item $n=2k$ for some odd $k \in \Psi_\Q(1)$, 
        \item $n=3k$ for some $k\in \Psi_\Q(1)$ which is divisible by $3$, or
        \item $n=28$.
    \end{itemize}
\end{theorem}

\begin{theorem}\label{ell>3}
    Let $p>3$ be a prime number. There exists a non-CM elliptic curve with rational $j$-invariant and an $n$-isogeny defined over a degree $p$ extension of $\Q$ if and only if $n \in \Psi_\Q(1)$ or $n=p \cdot k$ for some $k \in \Psi_\Q(1)$ which is divisible by $p$.
\end{theorem}

%\begin{remark} Denote by $I_\Q(d)$ the set of primes in $\cup_{k \leq d} \Psi_\Q(k)$. It was proved by Najman in \cite[Theorem 1.1.]{NAJMAN_2018} that $I_\Q(d)=I_\Q(1)$ for $d \leq 7$. Those prime degrees are $\{2,3,5,7,11,13,17,37\}$. Thus, the only degrees of isogenies that must be checked are products of these primes.
%\end{remark}

\begin{remark}
    Some of the proofs in this paper are aided by the computer algebra system Magma. The relevant code can be found in the following Github repository:
    \begin{center}
        \hyperlink{https://github.com/inova3c/Isogenies-over-cubic-fields-Magma-codes}{\texttt{github.com/inova3c/Isogenies-over-cubic-fields-Magma-codes}}
    \end{center}
\end{remark}

\section{Images of $\ell$-adic Galois representations}

Since the property of having an $n$-isogeny is invariant under quadratic twisting and any two non-CM elliptic curves with common $j$-invariant are quadratic twists of each other, it's not a loss of generality to only consider base changes of elliptic curves defined over $\Q$. 

For a number field $K$, let $G_K=\Gal(\overline{K}/K)$ denote the absolute Galois group of $K$, i.e. the group of all $\overline K$-automorphisms that fix $K$.

If $\phi:E \to E'$ is an $n$-isogeny defined over a number field $K$, then $\ker \phi \leq E(\overline{K})$ is a cyclic subgroup of order $n$ fixed by $G_K$. Conversely, for any cyclic subgroup of $E(\overline{K})$ of order $n$ fixed by $G_K$ there is an $n$-isogeny defined over $K$.

For a cyclic subgroup $C$ of $E(\overline{K})$, where $E/K$ is an elliptic curve, denote by $K(C)$ the smallest field extension $K'/K$ such that $G_{K'}$ fixes $C$. The field $K(C)$ is the \emph{field of definition} of $C$.

Throughout the paper, we will use the language of Galois representations of elliptic curves. For an elliptic curve $E$, we'll denote the mod $n$ Galois representation of $E$ by $\rho_{E,n}:G_\Q \to \GL_2(\Z/n\Z)$, and we'll denote the $\ell$-adic representation of $E$ by $\rho_{E, \ell^\infty}: G_\Q \to \GL_2(\Z_\ell)$. 

We will denote the subgroup of all uppertriangular matrices in $\GL_2(\Z/n\Z)$ by $B(n)$. The existence of a cyclic isogeny of degree $n$ is equivalent to the fact that the image of the mod $n$ representation is conjugate to a subgroup of $B(n)$.

The following proposition significantly simplifies our analysis. 

\begin{proposition}\cite[Proposition 3.3.]{Cremona-Najman}\label{nadneparnimnistanovo}
    Let $E/\Q$ be a non-CM elliptic curve. Let $\ell$ be an odd prime. If an elliptic curve has no $\ell$-isogenies defined over $\Q$, then it also has no $\ell$-isogenies defined over $K$ for any odd-degree number field $K$, unless $\ell=7$ and $j(E)=2268945/128$, in which case there is a $7$-isogeny defined over the cubic field $\Q(\alpha)$, where $\alpha^3-5\alpha-5=0$.
\end{proposition}

This means that nothing new happens over odd degree extensions in terms of odd prime degrees of isogenies, except in one special case. We now turn to prime power degrees of isogenies.

\subsection{Growth of cyclic isogeny degree for prime powers} \phantom{a} \\
\noindent Suppose that $E$ is an elliptic curve over a number field $K$. Let $\ell$ be a prime number and let $C$ be a cyclic subgroup of $E(\overline{K})$ of order $\ell^{n}$. The quantity of our interest is the index $[K(C):K(\ell C)]$. 

\begin{definition}
We say that the $\ell$-adic representation $\rho_{E, \ell^\infty}$ of $E$ is defined modulo $\ell^n$ if the image $\rho_{E, \ell^\infty}(G_\Q)$ contains all matrices congruent to $I$ mod $\ell^n$.
\end{definition}

\begin{lemma}{\cite[Proposition 3.7.]{Cremona-Najman}}\label{definedmodpn}
    Let $E$ be an elliptic curve defined over a number field $K$ such that its $\ell$-adic representation is defined modulo $\ell^{n-1}$ for some $n \geq 1$. Then for any cyclic subgroup $C$ of $E(\overline K)$ of order $\ell^n$, we have $[K(C):K(\ell C)]=\ell$.
\end{lemma}

Thus, if the $\ell$-adic representation is defined modulo $\ell^{n-1}$ and we're given an $\ell^{n-1}$-isogeny, we need to take a field extension of degree $\ell$ to get an $\ell^n$-isogeny.

In general, let $C$ be a cyclic subgroup of $E(\overline{K})$ of order $\ell^n$. 

\begin{lemma}
    There are exactly $\ell$ cyclic subgroups $D$ of $E(\overline K)$ such that $\ell D=C$.
\end{lemma}

\begin{proof}
    Let $P$ be a generator of $C$. There are $\ell^2$ points $Q\in E(\overline{K})$ with $\ell Q=P$. If $Q_0$ is one of the points, the set of all such points is $\{Q_0+T \mid T \in E[\ell]\}$.

    We claim that there are exactly $\ell$ points $T$ satisfying $Q_0+T \in \langle Q_0 \rangle$. Indeed, look at the equation $aQ_0=Q_0+T$ for some $a \in \Z/\ell^{n+1}\Z$. Since $T$ has order dividing $\ell$ and $Q_0$ has order $\ell^{n+1}$, this equation has a solution $T$ if and only if $\ell^n \mid a-1$, so there are exactly $\ell$ choices for $a$.

    From here, it follows that there are exactly $\ell$ cyclic subgroups $D$ satisfying $\ell D=C$,  each of them containing exactly $\ell$ solutions $Q$ to the equation $\ell Q=P$.
\end{proof}

Let $\mathcal S=\{D_0, D_1, \ldots, D_{\ell-1}\}$ be the set of subgroups satisfying $\ell D=C$. The absolute Galois group $G_K$ acts on $\mathcal S$. 

If the action is transitive, then the stabilizer of any $D_i$ has index $\ell$ in $G_K$. The corresponding fixed field has degree $\ell$ over $K$, and $D_i$ is defined over this field. This gives us an $\ell^{n+1}$-isogeny over a degree $\ell$ extension of $K$. 

If $G_K$ fixes any given group $D_i$, then $D_i$ is defined over $K$ and we have an $\ell^{n+1}$-isogeny defined over $K$. 

Let $Q$ be a point satisfying $\ell Q=P$. Let $S$ be a point such that $\{Q,S\}$ is a basis for $E[\ell^{n+1}]$. Then the $\ell$ subgroups can be represented as $$D_0=\langle Q \rangle, D_1=\langle Q+\ell^nS\rangle, \ldots, D_{\ell-1}=\langle Q+(\ell-1)\ell^n S \rangle.$$

Now take $\sigma \in G_K$. Since $\ell Q=P$ and $\sigma(P) \in \langle P \rangle$,   $\rho_{E,\ell^{n+1}}(\sigma)$ is of the form $\sm{a & b \\ \ell^n c & d}$. If we calculate the action of $\sigma$ on $Q+t\cdot \ell^nS$, we get  $$\sigma(Q+t\ell^n S)=(a+t\cdot b\ell^n)Q+(c\ell^n+td\ell^n)S.$$

Thus, $\sigma(Q+t\cdot \ell^nS)\in \langle Q+ \frac{c+td}{a}\ell^n S\rangle$. In other words, $$\sigma(D_t)=D_{\frac{c+td}{a}}.$$ 

Denote by $\AGL_1(\F_\ell)$ the group of all linear transformations on $\F_\ell$, i.e. all maps of the form $x \mapsto ax+b$ for $a \in \F_\ell^*$, $b \in \F_\ell$. Then the action of $G_K$ on $\mathcal S$ is equivalent to the action of a subgroup of $\AGL_1(\F_\ell)$ on $\F_\ell$.

We have the following elementary result, taken from \cite[Lemma 3.5.]{Cremona-Najman}.

\begin{lemma}
    For every subgroup $H \leq \AGL_1(\F_\ell)$ with $H\neq 1$, one of two cases occurs: either $H$ has order divisible by $\ell$, acts transitively on $\F_\ell$ and has a subgroup of index $\ell$; or is cyclic of order dividing $\ell-1$, acts on $\F_\ell$ with orbits of size $\ell-1$ with exactly one fixed point.
\end{lemma}

\begin{proof}
    We can represent the element $\{x \mapsto ax+b\}$ as a matrix $\sm{ a & b \\ 0 & 1}$. If the subgroup $H$ contains an element of the form $\sm{1 & b \\ 0 & 1}$ with $b \neq 0$, then $H$ acts transitively and the mentioned element has order $\ell$. The stabilizer of any $x \in \F_\ell$ is a subgroup of index $\ell$ . 

    If $H$ doesn't contain such elements, then for any $a$, there can be at most one $b$ with $\sm{a & b \\ 0 & 1}\in H$. This means that there is an embedding from $H$ to $\F_\ell^*$ given by $\sm{ a & b \\ 0 & 1} \mapsto a$. Since $\F_\ell^*$ is cyclic of order $p-1$, the claim follows. 

    Now suppose that $\sm{a & b \\ 0 & 1}$ is a generator of $H$. Then $x$ is a fixed point if and only if $x=ax+b$. There is a unique solution to this equation since $a \neq 1$, so there is a unique fixed point.
\end{proof}

Applying this to our situation yields the following result. This is a slightly restated version of \cite[Proposition 3.6.]{Cremona-Najman}.

\begin{lemma}\label{p-ili-dijeli-p-1}
    Let $\ell$ be a prime, $E$ an elliptic curve over a number field $K$ and $C$ a cyclic subgroup of $E(\overline K)$ of order $\ell^n$. Denote by $\mathcal S$ the set of all cyclic subgroups $D \subseteq E(\overline K)$ satisfying $\ell D=C$. Then:
    \begin{enumerate}[label=\textnormal{(\roman*)}]
        \item For all $D \in \mathcal S$, the index $[K(D):K(C)]$ either equals $\ell$ or divides $\ell-1$.
        \item If $[K(D):K(C)]=\ell$ for some $D$, then $[K(D):K(C)]=\ell$ for all $D$. Otherwise, there exists a subgroup $D \in \mathcal S$ with $K(D)=K(C)$. 
    \end{enumerate}
\end{lemma}

\begin{corollary}\label{l-kl}
    Let $\ell$ be a prime number and let $k$ be a positive integer divisible by $\ell$. Suppose that $E$ is an elliptic curve over a number field $K$ with a $k$-isogeny over $K$. Then $E$ has an $\ell \cdot k$-isogeny over $K$ or over a degree $\ell$ extension of $K$.
\end{corollary}

\begin{proof}
    It suffices to prove this when $k$ is a power of $\ell$. Let $C$ be a subgroup of order $\ell^n$ defined over $K$. Then, by Lemma \ref{p-ili-dijeli-p-1}, there exists a cyclic subgroup $D$ of order $\ell^{n+1}$ such that $[K(D):K]=\ell$ or $K(D)=K$. 

    The isogeny whose kernel is $D$ has the desired property.
\end{proof}

\subsection{Subgroups with uppertriangular tendencies}

\begin{remark}
    For a positive integer $m$, let $G_m$ denote the subgroup of $\GL_2(\Z_\ell)$ of all matrices of the form $\m{ a & b \\ \ell^mc & d}$, i.e. all matrices which are uppertriangular modulo $\ell^m$. From the proof of \cite[Proposition 3.6.]{Cremona-Najman} it follows that we can pick a basis for the $\ell$-adic Tate module of $E$ in which we have $$[K(C):K(\ell C)]=[\rho(G_{K(\ell C)}):\rho(G_{K(C)})]=[\rho(G_K)\cap G_{n-1}:\rho(G_K) \cap G_n],$$ where $\rho=\rho_{E, \ell^\infty}$ is the $\ell$-adic representation of $E$. 
\end{remark}

From Lemma \ref{p-ili-dijeli-p-1} and the above remark, we obtain the following corollary.

\begin{corollary}\label{slucaj-jednakosti}
     Let $\ell$ be a prime and let $m$ and $n$ be positive integers with $n\geq 2$ and $m<n$. Let $E$ be an elliptic curve over a number field $K$ and $C$ a cyclic subgroup of $E(\overline K)$ of order $\ell^n$. Then $K(C)=K(\ell^mC)$ if and only if there exists a basis for the $\ell$-adic Tate module of $E$ such that the following condition is satisfied: \begin{center} for any matrix $\sabcd \in \rho_{E, \ell^\infty}(G_K)$, if $c \equiv 0 \pmod{\ell^{n-m}}$ then $c \equiv 0 \pmod{\ell^n}$. \end{center}
\end{corollary}

This motivates the following definition.

\begin{definition}
     Let $\ell$ be a prime number and let $0 \leq m<n$  be integers. We say a group $H\leq \GL_2(\Z/\ell^n\Z)$ has \emph{$m$-uppertriangular tendencies} if the following condition is satisfied: \begin{center}
        for any matrix $\sabcd \in H$, if $c \equiv 0 \pmod{\ell^m}$ then $c \equiv 0 \pmod{\ell^n}$.
    \end{center}
\end{definition}
\begin{remark}
    To abbreviate, we'll call subgroups with $m$-uppertriangular tendencies "$m$-UTT subgroups" or just "$m$-UTT", and we'll call the defining property "$m$-UTT property". Also, we'll say "UTT subgroups" instead of "$1$-UTT subgroups". Note that "$0$-UTT" just means that every matrix in the subgroup is uppertriangular.
\end{remark}

We can rewrite Corollary \ref{slucaj-jednakosti} in new terms as follows.

\begin{corollary} \label{rewritten-as-utt}
    Let $\ell$ be a prime and let $m<n$ be positive integers with $n\geq 2$. Let $E$ be an elliptic curve over a number field $K$ and $C$ a cyclic subgroup of $E(\overline K)$ of order $\ell^n$. 
    
    Then $[K(C):K(\ell^m C)]=1$ if and only if there exists a basis for the $\ell$-adic Tate module of $E$ such that the image of the mod $\ell^n$ representation of $E$ is an $(n-m)$-UTT subgroup of $\GL_2(\Z/\ell^n\Z)$.
\end{corollary}

\subsubsection{Upper bounds on the size of $UTT$ subgroups} \phantom{a} \\
\noindent Let $m<n$. It turns out that an $m$-UTT subgroup $H\leq \GL_2(\Z/p^n\Z)$ is rather small compared to its image under the canonical surjection $\pi:\GL_2(\Z/p^n\Z)\to \GL_2(\Z/p^m\Z)$.

\begin{proposition}\label{3parametra}
    Let $k\leq m<n$ be positive integers and let $\pi:\GL_2(\Z/p^n\Z)\to \GL_2(\Z/p^m\Z)$ denote the canonical surjection. If $H\leq \GL_2(\Z/p^n\Z)$ is $m$-UTT and contains an element of the form $\sm{a & b \\ p^{k-1}c & d}$ with $c \neq 0 \pmod p$, then $|H|\leq p^{2n-2m+2k-2}|\pi(H)|$.
\end{proposition}

\begin{proof}
    Take a matrix $h_0=\sm{a_0 & b_0 \\ p^{k-1}c_0 & d_0}\in \pi(H)$ with $c_0 \not \equiv 0 \pmod{p}$.  Since the preimage of each element of $\pi(H)$ has the same cardinality, it suffices to prove that there are at most $p^{2n-2m+2k-2}$ matrices $h \in H$ with $\pi(h)=h_0$. We'll call such matrices $h$ \emph{lifts} of $h_0$. 

    Consider any lift $h=\sm{a & b \\ p^{k-1}c & d} \in H$ of $h_0$. Let $h'=\frac{1}{\det h}\sm{d' & -b' \\ -p^{k-1}c' & a'} \in H$ be any matrix with $\pi(h')=h_0^{-1}$. 

    Consider the bottom left entries of $hh'$ and $h'h$. They must be $0 \pmod{p^m}$, so by the UTT property they are also $0 \pmod{p^n}$. This yields the following relations: \begin{align*}
        p^{k-1}(cd'-c'd)&\equiv 0 \pmod{p^n}, \\
        p^{k-1}(ca'-c'a)&\equiv 0 \pmod{p^n}.
    \end{align*}
    This means that $a$ and $d$ are uniquely determined modulo $p^{n-k+1}$ by $c$ and $h_0$. Thus, there are at most $p^{k-1}$ choices for $a$ and $d$ once we choose $c$. Furthermore, there are at most $p^{n-m}$ choices for $b$ and $c$ which lift $b_0$ and $c_0$. In total, $h_0$ has at most $p^{2n-2m+2k-2}$ lifts.
\end{proof}

\begin{definition}
    Let $k$ and $n$ be positive integers with $k\leq n$. We say a subgroup $H \leq \GL_2(\Z/p^n\Z)$ is uppertriangular mod $p^k$ if every $\sabcd \in H$ satisfies $c \equiv 0 \pmod{p^k}$.
\end{definition}

Taking $k=m$ in Proposition \ref{3parametra} yields the following consequence.

\begin{corollary}
    Let $m<n$ be positive integers and let $\pi:\GL_2(\Z/p^n\Z)\to \GL_2(\Z/p^m\Z)$ denote the canonical surjection. If $H\leq \GL_2(\Z/p^n\Z)$ is $m$-UTT and $H$ isn't uppertriangular mod $p^m$, then $|H|\leq p^{2n-2}|\pi(H)|$.
\end{corollary}

If the group $H$ contains matrices of some particular shape, we can strengthen the claim. This seems like a slight improvement, but if $m$ and $k$ are small compared to $n$, this is actually an improvement by a factor of the order of magnitude $p^n$.

\begin{proposition}
    Let $k\leq m<n$ be positive integers with $n\geq m+2k-2$ and let $\pi:\GL_2(\Z/p^n\Z)\to \GL_2(\Z/p^m\Z)$ denote the canonical surjection. Suppose that $H\leq \GL_2(\Z/p^n\Z)$ is $m$-UTT. If $H$ contains an element of the form $\sm{a & b \\ p^{k-1}c & d}$ with $c, a+d \neq 0 \pmod p$, then $|H|\leq p^{n-m+4k-4}|\pi(H)|$.
\end{proposition}
\begin{proof}
    Let $h=\sm{a & b \\ p^{k-1}c & d}$ with $c, a+d \neq 0 \pmod p$. Let $h_0=\pi(h)$. As in the previous proof, it suffices to prove that $h_0$ has at most $p^{n-m+4k-4}$ lifts to $H$. From the proof of Proposition \ref{3parametra} it follows that $a$ and $d$ are uniquely determined modulo $p^{n-k+1}$ by $c$ and $h_0$. 

    Consider $h^2=\sm{a^2+p^{k-1}bc & * \\ p^{k-1}c(a+d) & *}$. Since $a+d \not \equiv 0 \pmod{p}$, we know that $a^2+p^{k-1}bc$ is uniquely determined modulo $p^{n-k+1}$. But this implies that $b$ is uniquely determined modulo $p^{n-2k+2}$. Thus, there are at most $p^{2k-2}$ possible choices for $b$ once we've chosen $a,c,d$. We also know there are at most $p^{n-m}$ choices for $c$ and at most $p^{k-1}$ choices for both $a$ and $d$. This implies there are at most $p^{n-m+4k-4}$ lifts of $h_0$ and the claim is proven.
\end{proof}

\begin{corollary} \label{m=1, k=1 version}
    Let $H\leq \GL_2(\Z/p^n\Z)$ be a UTT subgroup and denote by $\pi:\GL_2(\Z/p^n\Z)\to \GL_2(\Z/p\Z)$ the canonical surjection. Then $|H|\leq p^{2n-2}|\pi(H)|$ if $H$ is not uppertriangular mod $p$. Furthermore, if $H$ contains a matrix whose trace and bottom left entry are non-zero mod $p$, then $|H|\leq p^{n-1}|\pi(H)|$
\end{corollary}

\begin{proof}
    Take $k=m=1$ in the previous propositions.
\end{proof}

We'll have to deal with the case when we have a $p$-isogeny over a base field, and a $p^3$-isogeny over an extension of degree $p$. The bounds we've obtained so far are not suitable for this case, so we need another result.

For $p$-adic numbers $x$ and $y$, write $y=x+\mathcal O(p^k)$ if $y-x \in p^k \Z_p$. 

\begin{lemma}\label{padskirazvoj}
    Let $p>2$ be a prime number. Let $A=\sm{a & pb \\ pc & d}\in \GL_2(\Z_p)$ be a matrix such that $A \equiv I\pmod{p}$. Let $A^p=\sm{ a_p & pb_p  \\ pc_p & d_p}$. Then $pc_p=p^2c+\mathcal O(p^3)$.
\end{lemma}
\begin{proof}
    Write $A=I+pB$ for $B \in \mathcal M_2(\Z_p)$. Then, by the binomial theorem, we have $$A^p=I+p^2B+\mathcal O(p^3),$$
    so $pc_p=p^2c+\mathcal O(p^3)$. 
\end{proof}

\begin{proposition}\label{onasituacija}
    Let $n\geq 3$ and let $H$ be a $2$-UTT subgroup of $\GL_2(\Z/p^n\Z)$ which is uppertriangular modulo $p$. Denote by $\pi:\GL_2(\Z/p^n\Z)\to \GL_2(\Z/p\Z)$ the canonical surjection. Then any element of $\pi^{-1}(I)$ is uppertriangular mod $p^n$. Consequently, $|H|$ divides $p^{3n-3}|\pi(H)|$.
\end{proposition}
\begin{proof}
    Suppose that there is a matrix $h=\sm{a & pb \\ pc & d} \in H$ which is not uppertriangular and which satisfies $h \equiv I \pmod p$. Then by UTT property we have $c \not \equiv 0 \pmod{p}$. However, Lemma \ref{padskirazvoj} implies that the left most entry of $h^p$ is divisible by $p^2$ and isn't divisible by $p^3$, a contradiction.

    Thus, any matrix $h \in H$ with $\pi(h)=I$ is uppertriangular. Thus, the kernel of $\pi$ is a subgroup of the group of all uppertriangular lifts of $I$ in $\GL_2(\Z/p^n\Z)$. Since there are $p^{3n-3}$ such lifts, we conclude that $|\ker (\pi)|=|H|/|\pi(H)|$ divides $p^{3n-3}$.
\end{proof}

\section{Isogeny degrees over cubic extensions}

\subsection{Prime powers}

\subsubsection{Powers of two}

We now discuss cyclic isogenies whose degrees are powers of two over cubic fields. We start by connecting the previous chapter with images of $2$-adic Galois representations. 

\begin{corollary} \label{lema mod 2}
    Let $E /K$ be an elliptic curve over a number field. Let $n\geq 2$ be a positive integer and let $C$ be a cyclic subgroup of $E(\overline K)$ of order $2^n$. Suppose that the order of $\rho_{E, 2^n}(G_K)$ is divisible by $3$ and $K(C)=K(2^{n-1}C)$. Then $|\rho_{E, 2^n}(G_K)|\leq 3\cdot 2^n$.
\end{corollary}
\begin{proof}
   We see from Corollary \ref{rewritten-as-utt} that the mod $2^n$ image $\rho_{E, 2^n}(G_K)$ is a UTT subgroup of $\GL_2(\Z/2^n\Z)$.
   
   Furthermore, we claim that the mod $2$ image contains an element of order $3$. Namely, let $g$ be an element of order $3$ in $\rho_{E,2^n}(G_K)$. Then $\pi(g)$ has order $1$ or $3$, where $\pi:\GL_2(\Z/2^n\Z)\to \GL_2(\Z/2\Z)$ is the canonical surjection. 
     
   If $\pi(g)\equiv I \pmod{2}$, then the bottom-left entry of $g$ is even, so it must be $0$ mod $2^n$ by the UTT property. If we denote $g=\sm{a & b \\ 0 & c}$, then $g^3=I$ implies $a^3=c^3=1$, which implies $a=c=1$. Finally, direct calculation yields $b=0$ and $g=I$, a contradiction. Thus, $\pi(g)$ has order $3$.
    
    Elements of order $3$ in $\GL_2(\Z/2\Z)$ are $h=\sm{1 & 1 \\ 1 & 0}$ and $h^2=\sm{0 & 1 \\ 1 & 1}$, so the mod $2$ image contains both of them. In particular, it contains $\sm{1 & 1 \\ 1 & 0}$, a matrix with odd bottom left entry and odd trace, so the second claim from Corollary \ref{m=1, k=1 version} applies to our case. Since the mod $2$ image has at most $6$ elements, we obtain the bound $2^{n-1}\cdot 6=3\cdot 2^n$ on the mod $2^n$ image.
\end{proof}

In \cite{rouse-zureickbrown-2adic}, Rouse and Zureick-Brown classified all possible images of $2$-adic Galois representations of elliptic curves over $\Q$. They also provide a list of all possible images, which we use in the next proposition.  They also provided a webpage for each of the groups, which contains information on the rational points on the corresponding modular curves. The website can be found on  \begin{center}
    
\hyperlink{https://users.wfu.edu/rouseja/2adic/}{\texttt{https://users.wfu.edu/rouseja/2adic/}}.

\end{center}
From \cite[Corollary 1.3.]{rouse-zureickbrown-2adic}, we see that the index of the $2$-adic representation divides $64$ or $96$, and the $2$-adic representations are defined modulo $32$.

\begin{proposition}\label{potencije 2 detaljno}
    Let $E$ be a non-CM elliptic curve defined over $\Q$. Then $E$ has no cyclic $32$-isogenies defined over cubic extensions of  $\Q$. If $E$ has a cyclic $16$-isogeny or $8$-isogeny defined over a cubic extension of $\Q$, then this isogeny is already defined over $\Q$.
    If $E$ has a $4$-isogeny defined over a cubic extension of $\Q$ which isn't defined over $\Q$, then the image $\rho_{2,\infty}(G_\Q)$ of the $2$-adic representation of $E$ satisfies $\langle \rho_{2,\infty}(G_\Q), -I \rangle =H_{20}$, where $H_i$ denotes the $i$-th group of the list provided in \cite{rouse-zureickbrown-2adic}.
\end{proposition}

\begin{proof}  Without loss of generality we may assume that $-I$ belongs to the $2$-adic image. Otherwise, we may take a quadratic twist $E'$ of $E$ whose $2$-adic image does contain $-I$.

Suppose that $E$ has an $8$-isogeny over a cubic extension $K$ which isn't defined over $\Q$, and let $C$ be the kernel of the isogeny. We then have $K=\Q(C)$ by definition of $K$ and $\Q(C)=\Q(4C)$ as $[\Q(C):\Q(4C)]$ is a power of two which divides $3$. On the other hand $\rho_{E, 2^3}(G_K)$ is a subgroup of index $3$ of $\rho_{E,2^3}(G_\Q)$, which means that the order of $\rho_{E,2^3}(G_\Q)$ is divisible by $3$. Thus, we may apply Corollary  \ref{lema mod 2}, so the mod $8$ image has at most $24$ elements.

    This means that the index of the mod $8$ image in $\GL_2(\Z/8\Z)$ is at least $\frac{3 \cdot 2^9}{3 \cdot 2^3}=2^6$. There are $3$ possible $2$-adic images of index $64$, those are groups $H_{439}, H_{440}, H_{441}$ in Rouse and Zureick-Brown's list. However, none of them are defined modulo $8$, so their mod $8$ images have strictly smaller index than the $2$-adic images, so we get a contradiction. This proves the claim for $8$-isogenies.

    A similar argument proves the claim for $16$-isogenies. Namely, the mod $16$ image would have at most $48$ elements, which means the index of the $2$-adic representation would need to be at least $2^8$, which is impossible. 
    
    The non-existence of $32$-isogenies follows by the same argument, as there are no rational cyclic $32$-isogenies.

    Using the same reasoning for $4$-isogenies, we conclude that the mod $4$ image $\rho_{E,4}(G_\Q)$ has at most $12$ elements. The number of its elements is divisible by $6$ as it contains $-I$ and an index $3$ subgroup $\rho_{E,4}(G_K)$. 

    Thus, we can search through Rouse and Zureick-Brown's list for groups which have order $6$ or $12$ modulo $4$. We find three possible groups, $H_{20}, H_{21}$ and $H_{180}$. However, the corresponding modular curves $X_{21}$ and $X_{180}$ have no rational points as can be seen by looking at the corresponding webpages, so the image must be $H_{20}$. 
\end{proof}

We have the following immediate consequence.

\begin{lemma}
    The only powers of $2$ in $\Psi_\Q(3)$ are $\{1,2,4,8,16\}$.
\end{lemma}

It now makes sense to consider the group $H_{20}$ and the corresponding curve $X_{20}$. We can look at the webpage containing the info about this group on  \begin{center} \hyperlink{https://users.wfu.edu/rouseja/2adic/X20.html}{\texttt{https://users.wfu.edu/rouseja/2adic/X20.html}}
\end{center}
It turns out that $X_{20}$ is a genus $0$ curve which minimally covers curves $X_3$ and $X_7$ from the list. We have the following parametrizations of $X_3$, $X_7$, $X_{20}$:\begin{align}
    X_3: j&=-t^2+1728, \label{x3}\\
    X_7: j&=\frac{32t-4}{t^4}\label{x7},\\
    X_{20}: j&=\frac{32(t+1)(t^2-3)^3-4(t^2-3)^4}{(t+1)^4}\label{x20}.
\end{align}

\subsubsection{Powers of odd primes}

\begin{lemma}
    The only possible odd powers of primes in $\Psi_3$ are $\{1, 3, 3^2, 3^3, 5, 5^2, 7, 11, 13, 17, 37\}$.
\end{lemma}
\begin{proof}
     The squares of primes greater than $5$ are eliminated by the proof of \cite[Proposition 5.1.]{Borna}.

    For $p=5$, we can use \cite[Lemma 6.2.]{Borna} as we know that the existence of a $5$-isogeny over a cubic field implies the existence of a $5$-isogeny over $\Q$, so no powers of $5$ greater than $5^2$ are possible.

    Similarly, for $p=3$ we can use \cite[Lemma 7.2.]{Borna} to conclude that no powers of $3$ greater than $3^3$ are possible.

\end{proof}
Before solving the cases when $n$ has at least $2$ prime divisors, we need a more detailed analysis of powers of $3$.

\begin{lemma}\label{9-izo}
    Let $E/\Q$ be an elliptic curve with a $27$-isogeny over some cubic extension $K/\Q$. Then $E$ has a $9$-isogeny over $\Q$.
\end{lemma}
\begin{proof}
    Suppose that there exists a curve $E$ with a $27$-isogeny over $K$ and no $9$-isogeny over $\Q$. We know by Proposition \ref{nadneparnimnistanovo} that $E$ then has a $3$-isogeny over $\Q$. 
    
    We now use \cite[Corollary 1.3.1.]{rouse-elladic}. From there, we know that the index of $H:=\langle \rho_{E,27}(G_\Q), -I \rangle$ in $\GL_2(\Z/27\Z)$ is at most $36$, so $H$ has at least $3^7\cdot 4$ elements since $\GL_2(\Z/27\Z)$ has $16\cdot 3^9$ elements.

    However, if we choose an appropriate basis, $H$ has the property that each bottom left entry of a matrix in $H$ which is $0$ mod $9$ must also be $0$ mod $27$. In other words, $H$ is a $2$-UTT subgroup of $\GL_2(\Z/27\Z)$. Furthermore, $H$ is uppertriangular mod $3$ since there is a $3$-isogeny defined over $\Q$. Since there is no $9$-isogeny over $\Q$, we know that $H$ is not uppertriangular mod $9$.

    Denote by $\pi:\GL_2(\Z/27\Z) \to \GL_2(\Z/3\Z)$ the canonical surjection. By Proposition \ref{onasituacija}, $$3^7\cdot 4\leq |H|\leq 3^6\cdot |\pi(H)|\leq 3^6 \cdot |B(3)|=3^7\cdot 4.$$

    Thus, we must have equality everywhere.  This means that every uppertriangular matrix congruent to $I$ mod $3$ belongs to $H$, and that $\pi(H)$ consists of all uppertriangular matrices mod $3$.

    Now take a matrix $h=\sm{a & b \\ 3c & d}\in H$ with $a\equiv -d \pmod{3}$. Suppose that $h$ is not uppertriangular, i.e. that $c \not \equiv 0 \pmod 3$.
    
    Then $h^2=\sm{a^2+3bc & b(a+d) \\ 3c(a+d) & d^2+3bc}$ has bottom left entry divisible by $9$, so it must be divisible by $27$ and we have $a+d \equiv 0 \pmod{9}$. But then consider $\sm{1 & 0 \\ 0 & 4}\cdot h=\sm{a & b \\ 
    12c & 4d}$. This matrix is the same as $h$ modulo $3$, so we must have $a+4d \equiv 0 \pmod 9$, but this together with $a+d\equiv 0 \pmod 9$ implies $3d \equiv 0 \pmod{9}$, a contradiction. Thus, any matrix whose trace is $0$ mod $3$ only has uppertriangular lifts in $H$, and $H$ contains all of its uppertriangular lifts.

    We now conclude that any uppertriangular matrix is contained in $H$, since it's easy to check that matrices of trace zero along with $\pm I$ generate the subgroup $B(3)\leq \GL_2(\Z/3\Z)$. However, there are exactly $4\cdot 3^7$ uppertriangular matrices in $\GL_2(\Z/27\Z)$, so $H$ must be equal to the group of all uppertriangular matrices mod $27$, which is a contradiction since this would imply the existence of a $27$-isogeny over $\Q$.

\end{proof}

\subsection{Odd isogeny degrees}

\begin{lemma} \label{parovi-prostih}
    The only pairs of odd primes $(p,q)$ with $p<q$ and $pq \in \Psi_\Q(3)$ are $(3,5)$ and $(3,7)$.
\end{lemma}
\begin{proof}
    Suppose that for some odd primes $p<q$ we have a cyclic $pq$-isogeny of an elliptic curve $E/\Q$ defined over a cubic extension $K$. Then there is a $p$-isogeny defined over $\Q$ and a $q$-isogeny defined over $\Q$ by Proposition \ref{nadneparnimnistanovo}, so there is a $pq$-isogeny defined over $\Q$. The only pairs $(p,q)$ for which this is possible are $(3,5)$ and $(3,7)$.
\end{proof}

\begin{lemma}
    Let $n$ be an odd positive integer. Then $n \in \Psi_\Q(3)$ if and only if  $n \in \Psi_\Q(1) \cup \{27,45,63\}$.
\end{lemma}
\begin{proof}
    Corollary \ref{l-kl} implies $27, 45, 63 \in \Psi_\Q(3)$, as $9, 15, 21 \in \Psi_\Q(1)$.
    
    We now eliminate $27 \cdot 7$ and $27 \cdot 5$. For elliptic curves which fall under these cases, arguing as in Lemma \ref{parovi-prostih}, there is necessarily a $15$-isogeny or a $21$-isogeny defined over $\Q$.

    There are only four $j$-invariants for which there is a $15$-isogeny over $\Q$, and four $j$-invariants for which there is a $21$-isogeny over $\Q$. These can be found in Table 4 in \cite{alvaro}. 
    
    For each $j_0$ among those eight values, we look at the corresponding modular polynomial $\phi_{27}(x,j_0)$ and factor it into irreducibles using Magma. This polynomial has no irreducible factors of degree $1$ or $3$, so there is no cubic extension for which the corresponding elliptic curves have $9$-isogenies.

    In fact, in each of these cases, $\phi_{27}(x, j_0)$ factors into irreducibles of degrees $9$ and $27$, so the least possible degree of extension over which there is a $27$-isogeny is $9$.

    Finally, to eliminate $3 \cdot 25$, we check the modular polynomial $\phi_{25}(j_0)$ for values of $j_0$ for which there is a rational $15$-isogeny. There are again no irreducible factors of degree $3$.
\end{proof}

\subsection{Even isogeny degrees}
We now have a list of all odd elements of $\Psi_\Q(3)$. Now, for each odd $n \in \Psi_\Q(3)$, we need to find the largest integer $a$ for which $2^an\in \Psi_\Q(3)$.

First, we have the following lemma.

\begin{lemma}
    For every odd integer $n \in \Psi_\Q(1)$ we have $2n\in \Psi_\Q(3)$.
\end{lemma}

\begin{proof}
    Suppose that $E/\Q$ has a rational $n$-isogeny, where $n$ is odd. Consider $E$ in its Weierstrass form $$y^2=f(x),$$ where $f\in \Q[x]$ is a polynomial of degree $3$. Then either $f$ has a rational zero or $f$ is irreducible. Thus, $E$ either has a rational $2$-torsion point or a $2$-torsion point defined over a cubic extension of $\Q$. Since the existence of a $2$-torsion point over some extension is equivalent to the existence of a $2$-isogeny, $E$ has a $2$-isogeny defined over $\Q$ or over some cubic extension of $\Q$. Since $n$ is odd, we conclude that $E$ also has a $2n$-isogeny over the same extension.
\end{proof}

On the other hand, we have the following consequence of Proposition \ref{potencije 2 detaljno}.

\begin{corollary}
    For every odd integer $n>1$, we have $8n \not \in \Psi_\Q(3)$.  
\end{corollary}
\begin{proof}
    It suffices to prove the claim for odd primes $p \in \Psi_\Q(1)$. Suppose for the sake of contradiction that $E/\Q$ is an elliptic curve with an $8p$-isogeny. 
    
    From Proposition \ref{potencije 2 detaljno}, we know that the corresponding $8$-isogeny defined over a cubic field is actually defined over $\Q$. But we also know that a $p$-isogeny must be defined over $\Q$, so we obtain an $8p$-isogeny over $\Q$, which is impossible since $8p \not \in \Psi_\Q(1)$ for odd primes $p$.
\end{proof}

\begin{lemma}
    We have $27 \cdot 2, 9 \cdot 4 \in \Psi_\Q(3)$, and $45\cdot 2, 63 \cdot 2, 11 \cdot 4, 15\cdot 4, 17\cdot 4, 21\cdot 4, 37\cdot 4  \not \in \Psi_\Q(3)$.
\end{lemma}

\begin{proof} To prove $27 \cdot 2 \in \Psi_\Q(3)$ and $9 \cdot 4 \in \Psi_\Q(3)$ , we apply Corollary \ref{l-kl} to elliptic curves with $9\cdot 2$-isogenies and $3 \cdot 4$-isogenies over $\Q$.

    All of the remaining cases can be ruled out using modular polynomials, as there are only finitely many cases to consider. The $j$-invariants for which there is a cyclic isogeny of degree $15, 21, 11, 19$  or $37$  over $\Q$ can be found in Table 4 in \cite{alvaro}.

    To rule out $45 \cdot 2$ and $63 \cdot 2$, we need to look at the eight $j$-invariants $j_0$ which give us a $15$-isogeny or a $21$-isogeny. For each of those $j$-invariants, we need to check whether there exist a $2$-isogeny and a $9$-isogeny which are defined over the same cubic extension. In each of the $8$ cases, the polynomial $\phi_9(x, j_0)$ has only one irreducible factor $h(x)$ of degree $3$. 
    
    The $9$-isogenies which are defined over cubic fields are defined over fields of the form $\Q(t)$, where $t$ is any zero of $h(x)$. 
    
    The $2$-isogenies which are defined over cubic fields are defined over fields of the form $\Q(s)$, where $s$ is any zero of the polynomial $f(x)$ such that $y^2=f(x)$ is a short Weierstrass form equation for any elliptic curve $E/\Q$ with $j$-invariant $j_0$.

    Thus, if there are a $2$-isogeny and a $9$-isogeny defined over the same cubic field, then $\Q(t)=\Q(s)$ for some $t$ and $s$ satisfying $h(t)=0$ and $f(s)=0$. Then the splitting fields of $f(x)$ and $h(x)$ are the same. However, we can check using Magma that the splitting fields are different in each of the eight cases, so we can rule out $45\cdot 2$ and $63\cdot 2$. 

    To rule out $11\cdot 4, 15 \cdot 4, 19 \cdot 4, 21 \cdot 4, 37 \cdot 4$, we look at the polynomial $\phi_4(x, j_0)$ for each of the finitely many $j_0$. There are no irreducible factors of degree $1$ or $3$, which means there are no $4$-isogenies over cubic extensions.
\end{proof}

The numbers we haven't settled yet are $5\cdot 4$, $7\cdot 4$, $13\cdot 4$ and $27\cdot 4$. To do this, we will determine all rational points on some modular curves.

\begin{lemma}
    We have $5 \cdot 4, 13 \cdot 4, 27 \cdot 4 \not \in \Psi_\Q(3)$ and $7\cdot 4 \in \Psi_\Q(3)$.
\end{lemma}
\begin{proof}
    If there is an elliptic curve $E$ with a cyclic isogeny of degree $5\cdot 4, 7 \cdot 4$  or  $13 \cdot 4$ over a cubic field, it cannot have a $4$-isogeny over $\Q$ since that would give us a cyclic isogeny over $\Q$ of degree $5 \cdot 4, 7 \cdot 4$ or $13\cdot 4$. Similarly, if there is a  $27\cdot 4$-isogeny over a cubic extension, we must have a $9$-isogeny over $\Q$ by Lemma \ref{9-izo}. But then we cannot have a $4$-isogeny over $\Q$, as we would get a $36$-isogeny over $\Q$, which is impossible.

    In any case, there is a $4$-isogeny over a cubic extension, and there is no $4$-isogeny over $\Q$. Thus, by Lemma \ref{potencije 2 detaljno}, the $2$-adic image must equal $H_{20}$. 

    The corresponding modular curve $X_{20}$ is a genus zero curve, and there are maps from $X_{20}$ to $X_3$ and $X_7$. 

    Recall that those three curves are parametrized as follows:
    \begin{align}
    X_3: j&=-t^2+1728, \label{x3}\\
    X_7: j&=\frac{32t-4}{t^4}\label{x7},\\
    X_{20}: j&=\frac{32(t+1)(t^2-3)^3-4(t^2-3)^4}{(t+1)^4}\label{x20}.
\end{align}

Denote by $f_3, f_7, f_{20}$ the rational functions on the right hand side of those parametrizations.

The $j$-invariants for which there is a $p$-isogeny for $p \in \{3,5,7,13\}$ are parametrized as follows (this is taken from \cite{alvaro}, Table 3):

\begin{align}
    X_0(3): j&=\frac{(h+27)(h+3)^3}{h}, \\
    X_0(5): j&=\frac{(h^2+10h+5)^3}{h}, \\
    X_0(7): j&=\frac{(h^2+13h+49)(h^2+5h+1)^3}{h}, \\
    X_0(13): j&=\frac{(h^2+5h+13)(h^4+7h^3+20h^2+19h+1)^3}{h}.
\end{align}

Denote by $g_3, g_5, g_7, g_{13}$ the polynomials on the right hand side of those parametrizations.

We need to find all rational points on curves $g_p(h)=f_{20}(t)$ for $p \in \{3,5,7,13\}$, and check which of those points yields a non-CM $j$-invariant. We'll call these points non-trivial.

We need to prove there are no non-trivial points on those curves. It's enough to prove this for $f_7$ or $f_3$ instead of $f_{20}$, as there is a map from $X_{20}$ to $X_3$ and $X_7$.

\begin{itemize}
    \item \textbf{Case $p=3$.} We use the curve $X_7$, i.e. we look at the equation $g_3(h)=f_7(t)$. Using Magma, we see that this is a genus $1$ curve, and we can find all the rational points. There are $5$ points in total, three of which yield CM $j$-invariants. For the remaining two $j$-invariants, we use modular polynomials to check that there are no $9$-isogenies over $\Q$ (or equivalently that there are no $27$-isogenies over cubic extensions).

    \item \textbf{Case $p=5$.} We use the curve $X_7$ again, i.e. we look at the equation $g_5(h)=f_7(t)$. This is a genus $1$ curve, and we can find all points on this curve. It turns out that there are two points on this curve. However, one can check that the corresponding two $j$-invariants only yield $4$-isogenies over degree $6$ extensions of the rationals.

    \item \textbf{Case $p=7$.} The curve $g_7(h)=f_7(t)$ is hyperelliptic of genus two.  All rational points on this curve can be found using the Chabauty method.
    
    It turns out that there are two rational points, and the corresponding $j$-invariants are $\frac{351}{4}$ and $\frac{-38575685889}{16384}$. It can easily be checked using modular polynomials that both of them yield $4$-isogenies over cubic  extensions, so $28 \in \Psi_\Q(3)$.

    \item \textbf{Case $p=13$.} We use the curve $X_3$ this time, i.e. we look at the equation $g_{13}(h)=f_3(t)$. This is a genus one curve.  It turns out there are no non-trivial rational points on this curve, i.e. no suitable $j$-invariants, so $4 \cdot 13 \not \in \Psi_\Q(3)$.

\end{itemize}
\end{proof}

This concludes the proof of the cubic case (Theorem \ref{kubni}).

\begin{remark}
    The calculations in the previous proof were done in Magma.

    We would like to thank Borna Vukorepa for finding all rational points on the curve $g_7(h)=f_7(t)$ and providing the relevant Magma code.
    
\end{remark}

\section{Isogeny degrees over extensions of prime degree $p>3$}

 We've seen in Corollary \ref{l-kl} that $\Psi_\Q(p)$ contains all numbers of the form $k\cdot p$ for $k \in \Psi_\Q(1)$ divisible by $p$. We now prove that these are the only elements of $\Psi_\Q(p)$ which aren't already in $\Psi_\Q(1)$.

\begin{lemma}\label{no2or3overell}
    Suppose that $\ell$ and $p$ are prime numbers with $p \nmid \ell(\ell-1)$, and let $E/\Q$ be an elliptic curve with an $\ell^n$-isogeny defined over a degree $p$ extension $K/\Q$. Then this isogeny is defined over $\Q$.
\end{lemma}
\begin{proof}
    This follows from the fact that $[\rho_{E,\ell^n}(G_\Q):\rho_{E,\ell^n}(G_K)]$ divides $[K:\Q]$, and the fact that $\GL_2(\Z/\ell^n\Z)$ has order coprime to $p$. This means that the images $\rho_{E, \ell^n}(G_\Q)$ and $\rho_{E, \ell^n}(G_K)$ are equal, which implies the claim.
\end{proof}
\subsection{Case $p>5$}
\begin{lemma}\label{definiranomodl}
    Let $\ell\geq 7$ be a prime number, and let $E/\Q$ be an elliptic curve with an $\ell$-isogeny over a number field of odd degree. Then its $\ell$-adic representation $\rho_{E, p^\infty}(G_\Q)$ is defined modulo $\ell$. 
\end{lemma}
\begin{proof}
    Note that the $\ell$-isogeny is actually defined over $\Q$ by Proposition \ref{nadneparnimnistanovo}. The claim now follows immediately from \cite[Theorem 3.9.]{lombardo}.

\end{proof}

\begin{corollary}
   If $p\geq 7$ is a prime, then $p^3 \not \in \Psi_\Q(p)$. 
\end{corollary}
\begin{proof}
    This is immediate from Lemma \ref{definiranomodl}, Lemma \ref{definedmodpn} and the fact that there are no $p^2$-isogenies over $\Q$.
\end{proof}

\begin{lemma}\label{coprime-to-l}
    Let $p\geq 7$ be a prime, and let $n \in \Psi_\Q(p)$ be coprime to $p$. Then $n \in \Psi_\Q(1)$ and any $n$-isogeny defined over a degree $p$ extension of $\Q$ is already defined over $\Q$.
\end{lemma}
\begin{proof}
    It suffices to prove the claim for prime powers.  This follows from Lemma \ref{no2or3overell} as there are no primes $\ell \neq p$ in $\Psi_\Q(1)$ such that $p \mid \ell(\ell-1)$. 
\end{proof}

We can now finish the proof of Theorem \ref{ell>3} for $p \geq 7$. Namely, let $n \in \Psi_\Q(\ell)$, $n \not \in \Psi_\Q(1)$. By Lemma \ref{coprime-to-l}, $n=k\cdot p^t$ for some $k\in \Psi_\Q(1)$ coprime to $p$ and some $t \geq 1$. By Lemma \ref{definiranomodl}, $t \leq 2$.

Since $k$ and $p$ are coprime, there must be a $k \cdot p$-isogeny over $\Q$, so $k \cdot p \in \Psi_\Q(1)$. Thus, $n$ is of the form stated in the theorem.

\subsection{Case $p=5$}

\begin{lemma}\label{25-125}
     Let $E/\Q$ be an elliptic curve with a $125$-isogeny defined over a quintic extension of $\Q$. Then there is a $25$-isogeny defined over $\Q$.
\end{lemma}
\begin{proof}
    The curve $E$ certainly has a $5$-isogeny over $\Q$ by Proposition \ref{nadneparnimnistanovo}. From \cite[Theorem 2]{greenberg} it follows that either the $5$-adic representation is defined modulo $5$ or the index of the $5$-adic representation in $\GL_2(\Z_5)$ is not divisible by $25$. In the former case, we can use Lemma \ref{definedmodpn}.

    Suppose that $[\GL_2(\Z_5):\rho_{E,5^\infty}(G_\Q)]$ is not divisible by $25$. Suppose that there is no $25$-isogeny over $\Q$. Then the group $H:=\rho_{E,125}(G_\Q)$ is a $2$-UTT subgroup of $\GL_2(\Z/125\Z)$ which is uppertriangular mod $5$.

    Let $\pi: \GL_2(\Z/125\Z) \to \GL_2(\Z/5\Z)$ denote the canonical surjection. Proposition \ref{onasituacija} tells us that $|H| \text{ divides } 5^6 |\pi(H)|, \text{ which divides } 5^6 |B(5)|=5^7\cdot 16.$

    However, then $[\GL_2(\Z/125\Z): H]$ is divisible by $25$ as $|\GL_2(\Z/125\Z)|$ is divisible by $5^9$ and $|H|$ isn't divisible by $5^8$. This yields a contradiction.
\end{proof}

\begin{lemma}\label{potprostih-quintic}
    We have $5^4 \not \in \Psi_\Q(5)$.
\end{lemma}
\begin{proof}
    Let $E/\Q$ be an elliptic curve with a $5^4$-isogeny over a degree $5$ extension $K$. Then $E$ certainly has a $5$-isogeny over $\Q$ by Proposition \ref{nadneparnimnistanovo}.

    From \cite[Theorem 2.2]{greenberg}, we know that either the $5$-adic representation is defined modulo $5$ or the index is not divisible by $25$. In the former case, we immediately reach a contradiction with Lemma \ref{definedmodpn}.

    In the latter case, suppose that there is a degree $5$ extension $K/\Q$ such that there is a $5^4$-isogeny over $K$. Then, since $\rho_{E, 5^4}(G_K)$ is an index $5$ subgroup of $\rho_{E,5^4}(G_\Q)$, we know that the index $[\GL_2(\Z/5^4\Z):\rho_{E,5^4}(G_K)]$ is not divisible by $5^3$. 
    
    However, the index $[\GL_2(\Z/5^4\Z):B(5^4)]$ is divisible by $5^3$, so $\rho_{E,5^4}(G_K)$ cannot be contained in $B(5^4)$, a contradiction.
    %Lemmas \ref{definiranomodl} and \ref{definedmodpn} prove that there are no cyclic isogenies of degree $p^2$ for $p$ a prime greater than $5$.

    %Lemma \ref{no2or3overell} eliminates $2^5$ and $3^3$.

    %Now it remains to prove $5^4 \not \in \Psi_\Q(5)$ and $5^3 \in \Psi_\Q(5)$. Looking at Table 3 in \cite{rouse-elladic}, we see that all $5$-adic representations of elliptic curves over $\Q$ are defined modulo $25$. Thus, by Lemma \ref{definedmodpn}, any cyclic subgroup $C$ of order $5^4$ satisfies $[\Q(C):\Q(5C)]=[\Q(5C):\Q(25C)]=5$, so $[\Q(C):\Q]\geq 25$, which eliminates $5^4$.

    %Furthermore, taking any non-CM elliptic curve $E/\Q$ with a cyclic subgroup $C'$ of order $25$ fixed by $G_\Q$, and taking a cyclic subgroup $C$ of order $125$ with $5C=C'$, we obtain $[\Q(C):\Q]=[\Q(C):\Q(5C)]=5$, which means $125\in \Psi_\Q(5)$.
\end{proof}

\begin{lemma}
    If $n$ is coprime to $5$ and $n \in \Psi_\Q(5)$, then $n \in \Psi_\Q(1)$.
\end{lemma}

\begin{proof}
    It suffices to prove the claim for prime powers. 
    
    If $\ell^k$ is a prime power with $\ell\geq 7$, then the $\ell$-adic representation of $E$ is defined modulo $\ell$ by Lemma \ref{definiranomodl}. But then, using Lemma \ref{definedmodpn}, it follows that any isogeny of degree $\ell^k$ which isn't defined over $\Q$ has to be defined over an extension whose degree is a power of $\ell$, which is a contradiction.
    
    If $\ell^k$ is a prime power with $\ell<5$, the claim immediately follows from Lemma \ref{no2or3overell}.
\end{proof}

It remains to settle numbers divisible by $5$. Note that any $n \in \Psi_\Q(5)$ divisible by $5$ is either a power of $5$, of the form $2^a5^b$, or of the form $3^a5^b$.

The powers of $5$ have been dealt with already.

\begin{lemma}
    Let $n \in \Psi_\Q(5)$ and suppose $n=2^a5^b$ or $3^a5^b$ for positive integers $a$ and $b$. Then $a=1$ and $b<3$.
\end{lemma}
\begin{proof}
    If $n\in \Psi_\Q(5)$ is divisible by $20$ or by $45$, then, by Lemma \ref{no2or3overell}, there is a $20$-isogeny or a $45$-isogeny defined over $\Q$, which is impossible. Thus $a=1$.

    Now suppose that there is a $2\cdot 125$-isogeny or a $3\cdot 125$-isogeny over some cubic extension. Then Lemma \ref{25-125} and Lemma \ref{no2or3overell} imply the existence of a $2\cdot 25$-isogeny or a $3\cdot 25$-isogeny over $\Q$, which is a contradiction.
\end{proof}

This concludes the proof of Theorem \ref{ell>3}.

\bibliographystyle{abbrv}
\bibliography{lit}

\begin{thebibliography}{10}

\bibitem{nikola}
N.~Ad{\v{z}}aga, T.~Keller, P.~Michaud-Jacobs, F.~Najman, E.~Ozman, and B.~Vukorepa.
\newblock Computing quadratic points on modular curves {{\(X_0(N)\)}}.
\newblock {\em Math. Comput.}, 93(347):1371--1397, 2024.

\bibitem{Bourdon-Clark}
A.~Bourdon and P.~L. Clark.
\newblock Torsion points and isogenies on {CM} elliptic curves.
\newblock {\em J. Lond. Math. Soc., II. Ser.}, 102(2):580--622, 2020.

\bibitem{Cremona-Najman}
J.~E. Cremona and F.~Najman.
\newblock {{\(\mathbb{Q}\)}}-curves over odd degree number fields.
\newblock {\em Res. Number Theory}, 7(4):30, 2021.
\newblock Id/No 62.

\bibitem{greenberg}
R.~Greenberg.
\newblock The image of {Galois} representations attached to elliptic curves with an isogeny.
\newblock {\em Am. J. Math.}, 134(5):1167--1196, 2012.

\bibitem{Kenku}
M.~A. Kenku.
\newblock On the modular curves {$X_0(125), X_1(25), X_1(49)$}.
\newblock {\em Journal of the London Mathematical Society}, s2-23(3):415--427, 1981.

\bibitem{lombardo}
D.~Lombardo and S.~Tronto.
\newblock Some uniform bounds for elliptic curves over {{\(\mathbb{Q} \)}}.
\newblock {\em Pac. J. Math.}, 320(1):133--175, 2022.

\bibitem{alvaro}
A.~Lozano-Robledo.
\newblock On the field of definition of p-torsion points on elliptic curves over the rationals.
\newblock {\em Math.Ann. 357}, 2013.

\bibitem{mazur}
B.~Mazur.
\newblock Rational isogenies of prime degree. ({With} an appendix by {D}. {Goldfeld}).
\newblock {\em Invent. Math.}, 44:129--162, 1978.

\bibitem{rouse-elladic}
J.~Rouse, A.~V. Sutherland, and D.~Zureick-Brown.
\newblock {$\ell$}-adic images of {Galois} for elliptic curves over {$\Q$} (and an appendix with {John Voight}).
\newblock {\em Forum of Mathematics, Sigma}, 10, 2022.

\bibitem{rouse-zureickbrown-2adic}
J.~Rouse and D.~Zureick-Brown.
\newblock Elliptic curves over {{\(\mathbb {Q}\)}} and 2-adic images of {Galois}.
\newblock {\em Res. Number Theory}, 1:34, 2015.
\newblock Id/No 12.

\bibitem{Borna}
B.~Vukorepa.
\newblock Isogenies over quadratic fields of elliptic curves with rational $j$-invariant, 2022.

\end{thebibliography}

\end{document}